\documentclass{amsart}
\usepackage[margin=1.5in]{geometry}

\usepackage{amsfonts, amsmath, amssymb, amscd, amsthm, graphicx,enumerate,float}
\usepackage{hyperref,xypic}
\usepackage[usenames, dvipsnames]{color}
\usepackage{relsize}
\usepackage{exscale}
\usepackage{oubraces}
\usepackage{todonotes}
\usepackage{enumitem}
\usepackage{overpic}
\usepackage{caption}
\usepackage{subcaption}

\usepackage{thmtools}
\usepackage{thm-restate}
\usepackage[normalem]{ulem}

\newtheorem{thm}{Theorem}

\newtheorem{lem}[thm]{Lemma}
\newtheorem{prop}[thm]{Proposition}

\theoremstyle{definition}

\theoremstyle{remark}

\newtheorem{claim}[thm]{Claim}

\newfont{\eufm}{eufm10}

\renewcommand{\phi}{\varphi}

\newcommand{\Z}{\mathbb Z}

\newcommand{\Mod}{\operatorname{Mod}}

\newcommand{\A}{\mathcal {A}}

\newcommand{\map}{\mathrm{Mod}}

\newcommand{\Homeo}{\operatorname{Homeo}}
\newcommand{\sep}{\operatorname{Sep_2}}

\begin{document}

\title[Endperiodic loxodromics]{Constructing endperiodic loxodromics of infinite-type arc graphs }

\author{Priyam Patel}
\address{Department of Mathematics\\University of Utah\\Salt Lake City, UT 84112}
\email{patelp@math.utah.edu}

\author{Samuel J. Taylor}
\address{Department of Mathematics\\Temple University \\Philadelphia, PA 19122}
\email{samuel.taylor@temple.edu}

\date{}

\begin{abstract}
We give general conditions to produce endperiodic homeomorphisms that act loxodromically on various arc graphs of infinite-type surfaces.
\end{abstract}

\maketitle

\section{Introduction}

Let $S$ be an infinite-type surface with finitely many ends, possibly with boundary. For simplicity, we further assume that each component of $\partial S$ is compact, but allow for infinitely many boundary components. A homeomorphism $t \colon S \to S$ is \emph{endperiodic} if every end of $S$ is either \emph{attracting} or \emph{repelling} with respect to $t$ (see Section \ref{sec:background}). With $t$ fixed, the set of ends of $S$, denoted $e(S)$, is naturally partitioned into positive (attracting) 
and negative (repelling) ends.

A graph $\A(S)$ is called an \emph{arc graph} of $S$ if its vertices form a $\Mod(S)$--invariant 
collection of proper isotopy classes of essential simple closed curves and properly embedded arcs and where two vertices are joined by an edge if the isotopy classes can be realized disjointly. We further assume that some vertex of $\A(S)$ is either compact or non-separating, or more precisely that it is represented by such a curve or arc.
In any case, there is an action $\Mod(S) \curvearrowright \A(S)$ by isometries.

An essential finite-type subsurface $Y \subset S$ is said to be a \emph{witness} for $\A(S)$ if every vertex of $\A(S)$ intersects $Y$ essentially, i.e. the arc or curve associated to the vertex cannot be realized disjointly from $Y$. We further require that $Y$ separates 
the ends of $S$.
In this case, we have the partition $\partial Y =  \partial_- Y \sqcup \partial_0 Y \sqcup \partial_+ Y$, where each component of $\partial_0 Y$ cobounds a finite-type subsurface and each component of $\partial_\pm Y$ cobounds a $\pm$--end of $S$.

The purpose of this note is to prove the following:
\begin{thm}\label{thm:intro_main}
Let $S$ be an infinite-type surface with finitely many ends and let $t \colon S \to S$ be an endperiodic map.
Let $\A(S)$ be an arc graph of $S$ that contains a vertex which is either compact or non-separating,
and let $Y \subset S$ be
a witness for $\A(S)$. Finally, suppose that
\begin{enumerate}
\item both $\{t^n(\partial_+ Y)\}_{n\ge 1}$ and $\{t^{-n}(\partial_- Y)\}_{n\ge 1}$ are sequences of disjoint multicurves that are disjoint from $Y$ and exit the ends of $S$ as $n\to \infty$, and
\item both $t(\partial_- Y)$ and $t^{-1}(\partial_+ Y)$ intersect $Y$ essentially.
\end{enumerate}
Then for any $g \in \Mod(Y)$ such that 

\begin{enumerate}[resume]
\item $d_Y(t(\partial_- Y), t^{-1}(\partial_+ Y)) + 61 \le \ell_Y(g)$,
\end{enumerate}
the endperiodic map $f = tg$ acts loxodomically on $\A(S)$.
\end{thm}

Here, $d_Y$ denotes distance in the usual curve and arc graph $\A(Y)$ of $Y$ and 
$\ell_Y(g)$ denotes the (stable) translation length of $g$ on $\A(Y)$. Since $\ell_Y(g) > 0$ if and only if $g$ is pseudo-Anosov on $Y$, the theorem applies to any sufficiently large power of a pseudo-Anosov. The reader can check that the last item can be replaced with the inequality 
\begin{align}\tag{$\dagger$} \label{eq:3alt}
d_Y(t(\partial_- Y), t^{-1}(\partial_+ Y)) + 61 \le d_Y(t(\partial_- Y), gt(\partial_- Y)),
\end{align}
which does not require $g$ to be pseudo-Anosov (see the first paragraph in the proof of Theorem \ref{thm:intro_main}). 
For example, $g$ can be taken to be a power of a Dehn twist about a curve in $Y$ that is sufficiently far from $\pi_Y (t(\partial_- Y))$ in $\A(Y)$. That such endperiodic maps can be loxodromic was first observed in forthcoming work of Abbott, Miller, and the first author \cite{abbott2022dehn}. There it is shown that certain Dehn twists composed with a particularly simple endperiodic map (a standard shift map) act loxodromically on the \emph{relative arc graph} (see Section~\ref{sec:examples} for the definition of this graph). However, 
their results are disjoint from those in this note.

\smallskip
For examples of arc graphs that appear in the literature and for which our results apply, see Section \ref{sec:examples}. In particular, our construction includes the examples of \emph{strongly irreducible} (\cite[Definition 2.6]{field2021end}) endperiodic maps constructed by Field--Kim--Leininger--Loving \cite[Section 6.1]{field2021end}. 
Such maps are particularly important because they give rise to hyperbolic mapping tori whose convex core boundaries are totally geodesic \cite[Proposition 3.1]{field2021end}.

The motivation for this paper comes from one of the biggest open problems about infinite-type surfaces, which is to give a Nielsen--Thurston type classification of elements of the mapping class group. For finite-type $S$, the Nielsen--Thurston classification tells us that $f \in \Mod(S)$ is either periodic, reducible, or pseudo-Anosov. Moreover, the pseudo-Anosov elements are exactly those that act loxodromically on the curve graph $\mathcal{C}(S)$ \cite{MasurMinsky} and which give rise to hyperbolic mapping tori \cite{thurston1998hyperbolic}. The trichotomy of the Nielsen--Thurston classification does not exactly hold for infinite-type surfaces, so the goal is to understand what the analog of pseudo-Anosovs are in the infinite-type setting. One approach to this problem is to classify those elements of $\Mod(S)$ that act loxodromically on one of the relevant curve/arc graphs in this setting. For previous work on this topic, see \cite{bavard2016hyperbolicite, bavard2018gromov, abbott2021infinite, morales2020loxodromic, bar2021grand}.

In this paper, we give the first examples of endperiodic mapping classes acting loxodromically on the \emph{omnipresent arc graph} of Fanoni--Ghaswala--McLeay \cite{fanoni2020homeomorphic} and the \emph{grand arc graph} of Bar-Natan--Verberne \cite{bar2021grand}, which are the same graph when $S$ has finitely many ends, and on the separating curve graphs of Durham--Fanoni--Vlamis \cite{durham2018graphs}. 
A major novelty of our approach is that we give a very 
short proof that our elements act loxodromically on $\A(S)$, appealing to the lower bound on distance in $\A(S)$ given by Lemma~\ref{prop:lower_bound}. In this way, we avoid the ``starts like'' functions and the technical coding of arcs/curves on $S$ for lower bounds on distance as in \cite{bavard2016hyperbolicite, abbott2021infinite}.

\smallskip

\subsection*{Acknowledgments}
Patel was partially supported by NSF CAREER Grant DMS--2046889. Taylor was partially supported by NSF grant DMS--2102018 and the Sloan Foundation. The authors thank Federica Fanoni, Yvon Verberne, and Nicholas Vlamis for helpful discussions regarding Section~\ref{sec:examples} and Sanghoon Kwak, Tyrone Ghaswala, and Marissa Loving for comments on an earlier draft of the paper.

\section{Background and setup} 
\label{sec:background}
Throughout the paper $S$ is an infinite-type surface with finitely many ends, where $\partial S$ is a possibly empty collection of circles. The mapping class group $\Mod(S)$ of $S$ is the group of orientation-preserving homeomorphisms of $S$, up to isotopy. One could also restrict to the subgroup that permutes or fixes subsets of either $\partial S$ or the ends of $S$ without altering any results of the paper.

A homeomorphism $t \colon S \to S$ is \emph{end periodic} if each of its ends is either \emph{attracting} or \emph{repelling}. The end $e$ is \emph{attracting} if there is a $p\ge0$ and a neighborhood $U\subset S$ of $e$ such that $f^p(U) \subset U$ and $\bigcap_{k\ge 0}f^{kp}(U) = \emptyset$.
The end $e$ is \emph{repelling} if it is attracting for $t^{-1}$. See \cite{cantwell2021endperiodic} and \cite{field2021end} for details and examples.  
\smallskip

By a \emph{subsurface} $Z$ of $S$ we always mean a finite-type, essential subsurface. If $Z$ has genus $g$ and $p$ punctures/boundary components, then we further require that $3g+p-4\ge 0$. We next cover some basic material on subsurface projections to $Z$ as defined by Masur--Minsky \cite{MM2} and refer the reader there for more details. The reader will notice that we have ruled out the case where $Z$ is an annulus; this is because annuli do not satisfy the hypotheses of Theorem \ref{thm:intro_main}.

The arc graph $\A(Z)$ for $Z$ is the graph whose vertices are proper isotopy classes of essential simple closed curves (hereafter \emph{curves}) and essential proper arcs (hereafter \emph{arcs}) and whose edges correspond to vertices that can be realized disjointly in $Z$. 
There is a partially defined \emph{subsurface projection} $\pi_Z$ to $\A(Z)$ defined as follows. Let $\alpha$ be the isotopy class of a curve or arc. If $\alpha$ is disjoint from $Z$, or more precisely if it has a disjoint representative, then $\pi_Z(\alpha)$ is undefined. Otherwise, we say that $\alpha$ intersects $Z$ essentially. In this case, we realize $\alpha$ in minimal position with $\partial Z$ and take $\pi_Z(\alpha)$ to be the set of isotopy classes of arcs and curves in $\alpha \cap Z$. Observe that if $Z$ is a witness, then the map $\pi_Z \colon \A(S) \to \A(Z)$ is well-defined and $1$-Lipschitz in that sense that sets of diameter one are taken to sets of diameter one. Here and throughout, distance is taken with respect to the induced graph metric.

More generally, if $\alpha,\beta$ are curves or arcs on $S$ that intersect $Z$ essentially, then we use the notation
\[
d_Z(\alpha,\beta) : = \mathrm{diam}\left ( \pi_Z(\alpha) \bigcup \pi_Z(\beta) \right ).
\]
Of course, the above definitions also apply to \emph{multicurves} which are disjoint unions of essential simple closed curves. 

If $X$ and $Z$ are subsurfaces for which $\partial X$ intersects $Z$ essentially and $\partial Z$ intersects $X$ essentially, then $X$ and $Z$ are said to \emph{overlap}.

\smallskip
Finally, for any isometry $g$ of a metric space $X$, the \emph{(stable) translation length} of $g$ is given by
\[
\ell_X(g) = \lim_{k\to \infty} \frac{d_X(x,g^k x)}{k},
\]
for $x\in X$. This limit is well-defined and independent of $x$ by the triangle inequality. The isometry $g$ is said to be \emph{loxodromic} if $\ell_X(g) >0$.

\bigskip

Now returning to the setup of Theorem \ref{thm:intro_main}, define 
\[
\Lambda^+ = \bigcup_{i \in \Z}  f^i(\partial_-Y) \quad \text{and} \quad \Lambda^- = \bigcup_{i \in \Z}  f^i(\partial_+Y).
\]

\begin{lem}\label{lem:multicurve}
Both $\Lambda^+$ and $\Lambda^-$ are infinite multicurves of $S$. Moreover, $f^{-n}(\partial_-Y)$  exits the negative ends of $S$ and $f^{n}(\partial_+Y)$ exits the positive ends of $S$ as $n\to \infty$.
\end{lem}

\begin{proof}
By condition $(1)$ in Theorem \ref{thm:intro_main}, $t^n(\partial_+ Y)$ and $t^{-n}(\partial_- Y)$ are disjoint from the interior of $Y$ for all $n\ge0$.
Now using the fact that $g$ restricts to the identity on $S\setminus \mathrm{int}(Y)$, we have by induction that 
\begin{align}\label{ends}
f^n(\partial_+  Y) = t^n(\partial_+ Y) \; \text{ and } \; f^{-n}(\partial_-  Y) = t^{-n}(\partial_- Y) \quad \text{ for all } n\ge 0. 
\end{align}
Again using condition $(1)$ in Theorem \ref{thm:intro_main} implies that $ \bigcup_{i \ge 0}  f^i(\partial_+Y)$ and $\bigcup_{i \le 0}  f^i(\partial_-Y)$ are infinite multicurves. But then any pair of curves in one of $\Lambda^+$ or $\Lambda^-$ differs from a pair in these multicurves by translation by $f$, showing that $\Lambda^+$ and $\Lambda^-$ are also infinite multicurves.

The moreover statement now follows from eq. \ref{ends} and the corresponding statement for $t$ in the hypotheses of Theorem \ref{thm:intro_main}.
\end{proof}

The definition (and notation) for the multicurves $\Lambda^\pm$ is inspired by the Handel--Miller laminations associated to an endperiodic map. This perspective, however, is not required here. See \cite{cantwell2021endperiodic, LMT22} for details.

\section{Tools from Masur--Minsky}
To prove Theorem \ref{thm:intro_main}, we import a few central tools from Masur--Minsky theory into the setting of infinite-type surfaces. 

\subsection{Behrstock inequality and subsurface ordering}
We first recall the Behrstock inequality. If $X$ and $Z$ are overlapping finite-type subsurfaces of a surface $S$, and $\alpha$ is an arc or curve that intersects both $X$ and $Z$ essentially, then 

\begin{align} \label{Behrstock}
\min\{d_X(\alpha,\partial Z), d_Z(\alpha, \partial X)  \} \le 10. 
\end{align}

The first proof of the Behrstock inequality was given by Behrstock in \cite{behrstock2006asymptotic}, where the number $10$ was replaced by a constant depending on the complexity of the ambient surface; hence, his proof does not apply in the current (infinite complexity) setting. However, an explicit and elementary proof was recorded by Mangahas in \cite[Lemma 2.13]{mangahas2013recipe} (and attributed to Leininger) where an even sharper inequality than eq. \ref{Behrstock} was established. This proof does not make reference to the ambient surface and applies directly in our setting. Alternatively, one can set $S'$ to be any (finite-type) subsurface of $S$ that contains both $X$ and $Y$ and apply Mangahas's result directly to $\pi_{S'}(\alpha)$. In either case, eq. \ref{Behrstock} holds in our setting. In fact, since we consider projections to the arc (rather than curve) graph, the number $10$ can be reduced further, but pursuing this would distract from the goal at hand.

\medskip

The Behrstock inequality is a fundamental tool in defining an ordering on subsurfaces `between' two fixed curve systems. This ordering was first explicitly defined in \cite{behrstock2012geometry}, but we follow the simplified exposition of \cite{clay2012geometry}.

Assuming the setup above, fix $K\ge 20$ and let $\Omega(\Lambda^-,\Lambda^+, K)$ be the set of finite-type subsurfaces $Z$ such that $d_Z(\Lambda^-,\Lambda^+) \ge K$. Define $X\prec Z$ $\iff$ $X$ and $Z$ overlap and $d_X( \Lambda^-, \partial Z) \ge 10$. Clay--Leininger--Mangahas prove that $\prec$ is a strict partial ordering on $\Omega(\Lambda^-,\Lambda^+, K)$ \cite[Corollary 3.7]{clay2012geometry} and that any two overlapping subsurfaces $X$ and $Z$ in $\Omega(\Lambda^-,\Lambda^+, K)$ are ordered \cite[Proposition 3.6]{clay2012geometry}.

\subsection{A lower bound on distance in $\A(S)$}

This section outlines the key tool, giving a lower bound on distance in the arc graph. The proof is modeled on the effective lower bound for the Masur--Minsky distance formula given in \cite{ATW} (see also \cite[Lemma 10.1]{taylor2015right}) and a more general construction in \cite{sisto2019largest}.

\begin{prop}
\label{prop:lower_bound}
Let $\alpha, \beta \in \A(S)$ and let
$\Omega^\pitchfork(\alpha,\beta, 53)$ be any set of pairwise overlapping witness subsurfaces $Z$ such that $d_Z(\alpha, \beta) \ge 53$. Then, 
\[
\sum_{Z \in\Omega^\pitchfork(\alpha,\beta, 53)} d_Z(\alpha,\beta)    \le 5 \cdot d_{\A(S)}(\alpha,\beta).
\]
\end{prop}

\begin{proof}
Let $B=10$, set $K  = 5B+3$, and fix $\alpha, \beta \in \A(S)$ along with a geodesic $\alpha = \alpha_0, \alpha_1, \ldots, \alpha_N = \beta.$ For each $Y \in \Omega^\pitchfork(\alpha,\beta,K)$ choose $i_Y,t_Y \in \{0,\ldots,N\}$ as follows: $i_Y$ is the largest index $k$ with $d_Y(\alpha_0,\alpha_k) \le 2B+1$ and $t_Y$ is the smallest index $k$ greater than $i_Y$ with $d_Y(\alpha_k, \alpha_N) \le 2B+1.$  Write $I_Y = [i_Y, t_Y] \subset [0, N]$ and note that this subinterval is well-defined and that, since the projection of adjacent vertices in the geodesic have $d_Y$ less than or equal to $1$, $d_Y(\alpha_0,\alpha_{k}), d_Y(\alpha_{k},\alpha_N) \ge 2B+1$ for all $k \in I_Y$ and $d_Y(\alpha_{i_Y},\alpha_{t_Y}) \ge B+1$. 

The following claim implies that one cannot make large progress on distance simultaneously in overlapping subsurfaces.

\begin{claim}
If $Y,Z \in \Omega^\pitchfork(\alpha,\beta,K)$, then $I_Y \cap I_Z  = \emptyset $.
\end{claim}

\begin{proof}[Proof of claim]
Toward a contradiction, take $k \in I_Y \cap I_Z$. Since $Y$ and $Z$ overlap, the Behrstock inequality implies that either $d_Y(\alpha_0, \partial Z) \le B$ or $d_Z(\alpha_0, \partial Y) \le B$. Assume the former; the latter case is proven by exchanging the occurrences of $Y$ and $Z$ in the proof. By the triangle inequality, 
\begin{eqnarray*}
d_Y(\partial Z, \alpha_k) &\ge& d_Y(\alpha_0,\alpha_k) - d_Y(\alpha_0, \partial Z) \\
&\ge& 2B+1 - B \ge B+1.
\end{eqnarray*}

The Behrstock inequality applied again tells us that $d_Z(\partial Y,\alpha_k) \le B$. Therefore, 
\begin{eqnarray*}
d_Z(\partial Y, \alpha_N) &\ge& d_Z(\alpha_k,\alpha_N) - d_Z(\partial Y,\alpha_k) \\
&\ge& 2B+1 -B \ge B+1.
\end{eqnarray*}
The Behrstock inequality applied one last time tells us that $d_Y(\partial Z, \alpha_N) \le B$. This, together with our initial assumption, implies
$$d_Y(\alpha_0,\alpha_N) \le d_Y(\alpha_0, \partial Z) + d_Y(\partial Z, \alpha_N) \le 2B < K$$
contradicting the fact that $Y \in \Omega^\pitchfork(\alpha,\beta,K)$.
\end{proof}

Returning to the proof of the proposition, since $I_Y$, $I_Z \subset [0, N]$ do not overlap for $Y \neq Z$, $N =  d_{\A(S)}(\alpha, \beta) $ is at least the sum of the length of these intervals. That is,
\[
\sum_{Y \in \Omega^\pitchfork (\alpha,\beta,K)} | t_Y - i_Y | \le d_{\A(S)}(\alpha, \beta) .
\]

Finally, using that the subsurface projections are $1$--Lipschitz,
\begin{eqnarray*}
d_Y(\alpha,\beta) &\le& d_Y(\alpha_{i_Y}, \alpha_{t_Y}) + 4B +2 \\
&\le& |t_Y - i_Y| +4B +2.
\end{eqnarray*}
Since, for each $Y \in \Omega^\pitchfork(\alpha,\beta,K)$, $d_Y(\alpha,\beta) \ge 5B+3$ we have $$\frac{1}{5} \cdot d_Y(\alpha,\beta) \le d_Y(\alpha,\beta) - 4B - 2 \le |t_Y - i_Y|$$ and so putting this together with the inequality above gives
$$\sum_{Y \in \Omega^\pitchfork(\alpha,\beta,K)} d_Y(\alpha,\beta) \le 5 \cdot d_{\A(S)}(\alpha,\beta) $$
as required.
\end{proof}

\section{Proof of Theorem~\ref{thm:intro_main}} 

Recall that $\Lambda^+ = \bigcup_{i \in \Z}  f^i(\partial_-Y)$ and $\Lambda^- = \bigcup_{i \in \Z}  f^i(\partial_+Y)$.  These multicurves are invariant under $f$ by definition. The theorem follows from three more lemmas, the first of which relates the displacement of $g$ to the quantity $d_Y(\Lambda^+,\Lambda^-)$.

\begin{lem}\label{lem:dist_trans}
Assume items (1) and (2) of Theorem~\ref{thm:intro_main}. Then for any $g \in \Mod(Y),$
\[
\vert d_Y\left (t(\partial_- Y), g^{-1}t(\partial_- Y)\right) - d_Y(\Lambda^+,\Lambda^-) \vert  \le  d_Y(t(\partial_-Y), t^{-1}(\partial_+ Y)) + 2.
\]
 \end{lem}

\begin{proof}
 The triangle inequality gives  $$d_Y(t(\partial_- Y), g^{-1}t(\partial_- Y)) \leq d_Y(t(\partial_- Y), \Lambda^+) + d_Y(\Lambda^+,\Lambda^-) + d_Y(g^{-1}t(\partial_- Y), \Lambda^-).$$ First, $g$ acts trivially on $\partial_-Y$, so that $t(\partial_- Y) = f(\partial_- Y)$. Since $f(\partial_- Y)$ is contained in $\Lambda^+$, $d_Y(t(\partial_- Y), \Lambda^+)\le1$. 
 
 Second, 
$$d_Y(g^{-1}t(\partial_- Y), \Lambda^-) \leq d_Y(g^{-1}t(\partial_- Y), g^{-1}t^{-1}(\partial_+Y)) + d_Y(g^{-1}t^{-1}(\partial_+ Y), \Lambda^-).$$ However, $g^{-1}t^{-1}(\partial_+ Y) = f^{-1}(\partial_+ Y)$ which is in $\Lambda^-$. Hence, $d_Y(g^{-1}t^{-1}(\partial_+ Y), \Lambda^-) \le 1.$
Since $$d_Y(g^{-1}t(\partial_- Y), g^{-1}t^{-1}(\partial_+Y))= d_Y(t(\partial_-Y), t^{-1}(\partial_+ Y)),$$ we obtain
$$d_Y(g^{-1}t(\partial_- Y), \Lambda^-) \leq d_Y(t(\partial_-Y), t^{-1}(\partial_+ Y)) + 1.$$ Combining these two facts concludes the lemma. 
  \end{proof}

The next lemma states that no power of $f$ translates $Y$ off itself.

\begin{lem}\label{lem:overlap}
With $S$ and $Y$ as in the statement of Theorem~\ref{thm:intro_main}, assume that 
\[
d_Y(t(\partial_- Y), t^{-1}(\partial_+ Y)) + 22 \le d_Y(t(\partial_- Y), g^{-1}t(\partial_- Y)).
\]
Then $f^i(Y)$ overlaps $f^j(Y)$ for all $i \neq j$. 
\end{lem}

\begin{proof}
Lemma \ref{lem:dist_trans} together with the assumption implies that
\begin{align}\label{eq:d_big}
d_Y(\Lambda^-,\Lambda^+) &\ge d_Y(t(\partial_- Y), g^{-1}t(\partial_- Y)) -d_Y(t(\partial_- Y), t^{-1}(\partial_+ Y)) -2 \\
& \ge 20, \nonumber
\end{align}
and so $Y \in \Omega(\Lambda^-,\Lambda^+, 20)$. Since $\Lambda^\pm$ are $f$--invariant, we also have that $f^i(Y) \in  \Omega(\Lambda^-,\Lambda^+, 20)$ for all $i \in \mathbb{Z}$.

Note that $t(\partial_- Y) \subset t(\partial Y) = f(\partial Y)$ which intersects $Y$ essentially by condition $(2)$ of Theorem \ref{thm:intro_main}. Additionally, $t(\partial_-Y) = f(\partial_-Y)$ which is in $\Lambda^+$. Hence, $d_Y(f(\partial Y), \Lambda^+) \le 1$ and so $d_Y(f(\partial Y), \Lambda^-) \ge 20 - 1 \ge10$. This implies that $Y \prec f(Y)$. But then since $\Lambda^\pm$ are $f$--invariant, $f^i(Y) \prec f^{i+1}(Y)$ for all $i \in \mathbb{Z}$. Applying the fact that the partial order $\prec$ is transitive we see that $f^i(Y) \prec f^j(Y)$ for all $i<j$. This, in particular, means that $f^i(Y)$ and $f^j(Y)$ overlap for all $i\ne j$, completing the proof.
\end{proof}

The final lemma states, in particular, that for $k\gg 1$, $f^k(\alpha)$ is eventually close to $\Lambda^+$ and $f^{-k}(\alpha)$ is eventually close to $\Lambda^-$ as seen from the perspective of $Y$, at least for a specially chosen $\alpha \in \A(S)$.

\begin{lem}
\label{lem:nearby_exist}
For any $Y$ and $g$ as above, there is an $\alpha \in \A(S)$ and $k_0\ge0$
so that for any $k\ge k_0$,
\[
d_Y(f^k(\alpha), \Lambda^+) \le 3 \text{ and } d_Y(f^{-k}(\alpha), \Lambda^-) \le 3.
\]
\end{lem}

\begin{proof}
First suppose that there is an $\alpha \in \A(S)$ that is contained in a finite-type subsurface. Then there exists $k_1$ so that for all $k\ge k_1$, $f^{-k+1}(\partial_-Y)$ is disjoint from $\alpha$. This is due to the fact (see eq. \ref{ends}) that $f^{-n }(\partial_-Y) = t^{-n }(\partial_- Y)$ escapes $S$ as $n \to \infty$. Applying $f^k$ we have that $f(\partial_- Y)$ is disjoint from $f^k(\alpha)$ and each of these intersect $Y$. Since $f(\partial_-Y)$ is contained in $\Lambda^+$, $d_Y(\Lambda^+, f^k(\alpha)) \le 2$ as required. Similarly, there exists $k_2$ such that for all $k \ge k_2$, $f^{k+1}(\partial_+Y)$ is disjoint from $\alpha$. Therefore, $d_Y(\Lambda^-, f^{-k}(\alpha)) \le 2$. Letting $k_0 = \max \{k_1, k_2\}$ proves the lemma in this case.

Otherwise, let $\beta \in \A(S)$ be an (isotopy class of an) arc that is not contained in a finite-type subsurface. By the hypothesis on $\A(S)$, we may assume that $\beta$ is non-separating.
\begin{claim}
There is an $\alpha \in \A(S)$ joining the same ends as $\beta$ (at most one of which is at a boundary component of $S$) with the property that $\alpha$
intersects each connected component of the infinite multicurve 
\[
m = \bigcup_{i \le 0}  f^i(\partial_-Y)\quad \mathlarger{\bigcup} \quad \bigcup_{i \ge 0}  f^i(\partial_+Y)
\]
no more than twice.
\end{claim}
Note that the fact that $m$ is a multicurve comes immediately from eq. \ref{ends} and condition $(1)$ in Theorem \ref{thm:intro_main}.

\begin{proof}[Proof of claim]
We begin by defining a compact exhaustion by subsurfaces $K_0 \subset K_1 \subset K_2 \ldots$ of $S$ by setting  $K_i$ to be the closure of the (unique) bounded component of $S \setminus \left (f^i(\partial_+Y) \cup f^{-i}(\partial_-Y) \right )$. In particular, $K_0 = Y$, $K_{i} \subseteq K_{i+1}$, and $m = \bigcup_{i\ge0} \partial K_i$. Using the exhaustion $\{K_i\}_{i\ge0}$, we see that for any end $e$ of $S$ and any component of $\partial_\pm Y$ (or more generally, component of $m$) there is a simple ray starting at that component, exiting the end $e$, and crossing no component of $m$ more than once.

If $\beta$ is a properly embedded arc whose ends exit the ends $e_1, e_2$ of $S$ (where possibly $e_1=e_2$),
take two rays $r_1$ and $r_2$ starting on $\partial Y$, intersecting each component of $m$ at most once, and exiting $e_1$ and $e_2$, respectively. Then join $r_1$ and $r_2$ via an essential simple arc in $Y$ to produce a new proper arc $\alpha$. In the case where $\beta$ goes to and from the same end, i.e. $e_1 = e_2$, further assume that the rays $r_1$ and $r_2$ are disjoint, and join them via an essential simple arc in $Y$. In general, this could produce a separating arc $\alpha$, but $r_2$ can be chosen to guarantee that $\alpha$ is non-separating.

When $\beta$ has one endpoint on a boundary component of $S$, that boundary component is contained in $K_i$ for some $i$. Then a similar argument produces a proper arc $\alpha$ joining the same ends as $\beta$ that meets no component of $m$ more than once.

Since $\beta$ and $\alpha$ are non-separating and start and end at the same ends of $S$, they are in the same $\map(S)$--orbit and hence $\alpha \in \A(S)$. In more detail, $\overline{S\setminus \beta}$ and $\overline{S \setminus \alpha}$ are homeomorphic by the classification of surfaces with non-compact boundary (see \cite[Theorem 2.2]{brown1979classification} and \cite[Theorem 3.3]{dickmann2021mapping}) and the homeomorphism extends to one bringing $\beta$ to $\alpha$, just as is the case with non-separating simple closed curves.
\end{proof}

Finally, fix $k\ge 1$. Since $\alpha$ intersects each component of $f^{-k+1}(\partial_- Y)$ and $f^{k-1}(\partial_+ Y)$ at most twice, 
we must have that $f^k(\alpha)$ intersects each component of $f(\partial_-Y)$ at most twice and $f^{-k}(\alpha)$ intersects each component of $f^{-1}(\partial_+ Y)$ at most twice. Since $f(\partial_-Y) = t(\partial_-Y)$ and $f^{-1}(\partial_+Y) = g^{-1}(t^{-1}(\partial_+ Y))$, and $t(\partial_-Y)$, $t^{-1}(\partial_+Y)$ meet $Y$ essentially by assumption, each of the curves $f(\partial_- Y)$ and $f^{-1}(\partial_+ Y)$
meets $Y$ essentially. Hence,
\[
d_Y(f^k(\alpha), f(\partial_-Y)) \le 2 \text{ and } d_Y(f^{-k}(\alpha), f^{-1}(\partial_+ Y)) \le 2.
\] 
Here we have used the elementary fact that arcs in $\A(Y)$ that intersect at most twice have distance at most $2$. 
Now the proof follows from the fact that $f(\partial_-Y)$ is contained in $\Lambda^+$ and $f^{-1}(\partial_+ Y)$ is contained in $\Lambda^-$.
\end{proof}

With these lemmas proven, we can turn to the 
\begin{proof}[Proof of Theorem \ref{thm:intro_main}]
Let $\alpha \in \A(S)$ be as in Lemma \ref{lem:nearby_exist}. Since translation length satisfies $\ell_Y(g) = \ell_Y(g^{-1}) \le d_Y(t(\partial_- Y), g^{-1}(t(\partial_- Y)))$, condition $(3)$ implies that 
\[
d_Y(t(\partial_- Y), t^{-1}(\partial_+ Y)) + 61 \le d_Y(t(\partial_- Y), g^{-1}(t(\partial_- Y))),
\]
 which we had previously called eq. \ref{eq:3alt}.
 So by Lemma \ref{lem:overlap}, $f^i(Y)$ overlaps $f^j(Y)$ for $i\neq j$. As in eq. \ref{eq:d_big}, Lemma \ref{lem:dist_trans} implies that $d_Y(\Lambda^-,\Lambda^+) \ge 61-2 = 59$. 

By Lemma \ref{lem:nearby_exist}, there is a $k_0\ge0$ so that for all $k\ge k_0$,
\[
d_Y(f^k(\alpha), \Lambda^+) \le 3 \text{ and } d_Y(f^{-k}(\alpha), \Lambda^-) \le 3.
\]
Then if we fix $k\gg 1$, we have that for all $i$ in the range $-k +k_0\le i\le k-k_0$
\begin{align*}
d_{f^iY} (f^{-k}(\alpha), f^{k}(\alpha))&= d_Y (f^{-k-i}(\alpha), f^{k-i}(\alpha)) \\
&\ge d_Y(\Lambda^-,\Lambda^+) -6\\
&\ge 53.
\end{align*}

We conclude that for $-k +k_0\le i\le k-k_0$, $\{f^i(Y)\}$ is a set of pairwise overlapping witness subsurfaces
$\Omega^\pitchfork(f^{-k}(\alpha), f^k(\alpha),53)$, and so by Proposition \ref{prop:lower_bound} we have
\begin{align*}
5 \; d_{\A(S)}(f^{-k}(\alpha), f^k(\alpha)) &\ge \sum_{-k +k_0\le i\le k-k_0} d_{f^iY} (f^{-k}(\alpha), f^{k}(\alpha))  \\
&\ge  \sum_{-k +k_0\le i\le k-k_0} (d_Y(\Lambda^-,\Lambda^+) -6) \\
&\ge 2(k-k_0) \cdot  (d_Y(\Lambda^-,\Lambda^+) -6).
\end{align*}
Dividing both sides by $k$ and taking the limit as $k \to \infty$ shows that $\ell_{\A(S)}(f) >0$ and completes the proof.
\end{proof}

We remark that the proof establishes that for $Y \subset S$ satisfying conditions $(1)$ and $(2)$ in the statement of Theorem \ref{thm:intro_main}, the translation length $\ell_{\A(S)}(f)$ is bounded below by a linear function of $\ell_{\A(Y)}(g)$.

\section{Examples}
\label{sec:examples}

In this section, we discuss three
important graphs associated to infinite-type surfaces that can help us understand the appropriate analog of pseudo-Anosovs in the infinite-type setting. Recall that constructing and classifying those mapping classes acting loxodromically on these graphs is one of the promising approaches to an analog of the Nielsen-Thurston classification for infinite-type surfaces and is the primary goal of \cite{abbott2021infinite, bavard2016hyperbolicite, bavard2018gromov, morales2020loxodromic}.

Note that the curve graph for an infinite-type surface is finite diameter (in fact, diameter 2) and this necessitates the introduction of better suited graphs. 

\subsection{Relative Arc Graph}
The relative arc graph was introduced by Aramayona--Fossas--Parlier as a generalization of the ray graph defined by Calegari and studied by Bavard in \cite{bavard2016hyperbolicite}. Fixing a finite collection of punctures (isolated planar ends) $P$ on $S$, the vertices of the relative arc graph $\A(S, P)$ are isotopy classes of simple arcs starting and ending on $P$. There is an edge between two vertices if the isotopy classes of arcs admit disjoint representatives. In \cite{aramayona2017arc}, Aramayona--Fossas--Parlier prove that  $\A(S, P)$ is infinite-diameter and $7$-hyperbolic. 

In our setting, we blow up all punctures of $S$ to boundary components (so that they do not contribute to the ends of $S$), and given an endperiodic map $t$, we choose $Y$ to be any finite-type subsurface containing $P$ that satisfies condition (1) of Theorem~\ref{thm:intro_main}. In particular, such a surface $Y$ containing $P$ is \emph{nondisplaceable}, i.e. for all $f \in \Homeo(S)$, $f(Y) \cap Y \neq \emptyset$. Thus, $Y$ must satisfy condition (2) of Theorem~\ref{thm:intro_main}. Any subsurface that contains $P$ is a witness for the graph $\A(S, P)$. Now taking any $g \in \map(Y)$ satisfying condition (3) of Theorem~\ref{thm:intro_main} gives an endperiodic map $f = tg$ that acts loxodromically on $\A(S, P)$. 

For an example, let $S$ be the ladder surface (the surface with exactly two ends, both of which are accumulated by genus) and let $P$ consist of one puncture $p$ that we view as a boundary component. Let $t$ be the standard handleshift on $S$, shown in Figure~\ref{fig:shift}, which shifts the handles over by one to the right and fixes $p$ pointwise. Then, the subsurface $Y$ shown in Figure~\ref{fig:shift} satisfies conditions (1) and (2) of Theorem~\ref{thm:intro_main}. In this example, $t(\partial_- Y)$ and $ t^{-1}(\partial_+ Y)$ are disjoint and so $d_Y(t(\partial_- Y), t^{-1}(\partial_+ Y)) =1$.

\begin{figure}[h]
\begin{center}
\begin{overpic}[trim = 1.25in 8.3in 1in 1.45in, clip=true, totalheight=0.1\textheight]{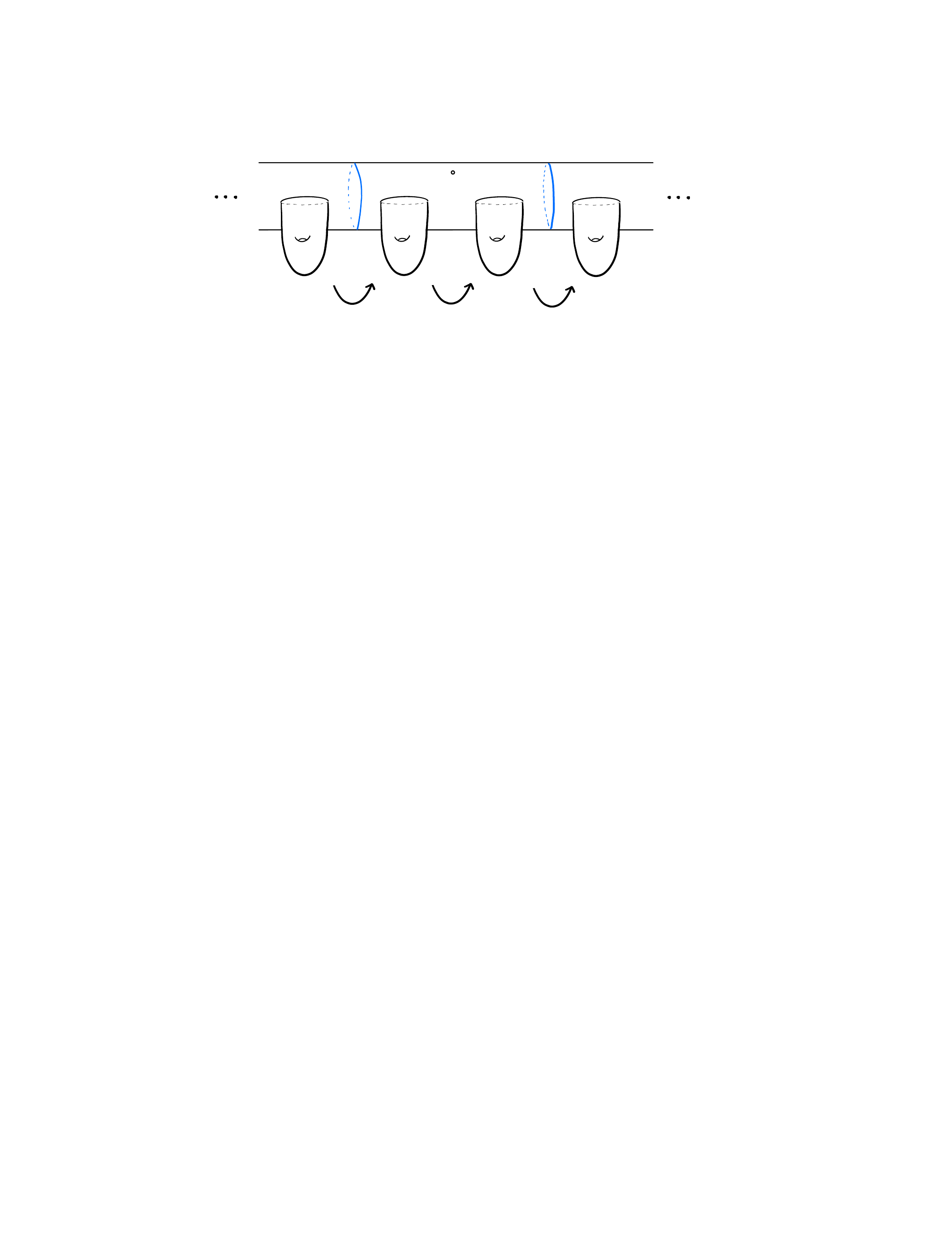}
\put(44,15){\large{\textcolor{blue}{$Y$}}}
\put(72,5){\large{$S$}}
\end{overpic}
\caption{}\label{fig:shift}
\end{center}
\end{figure}

Note that in the case where $P$ consists of exactly one puncture $p$, we obtain results similar to those of \cite{abbott2021infinite}, where the authors construct an infinite family of elements acting loxodromically on $\A(S,p)$, which can be expressed as the composition of a very simple endperiodic map on $S$ with a pseudo-Anosov homeomorphism on a surface $Y$ containing $p$. The key difference is that our results give a more general construction, allowing one to vary the element $g \in \map(Y)$ subject to the distance condition (3) in Theorem~\ref{thm:intro_main}, and our argument does not require the technical coding of arcs used in \cite{abbott2021infinite}.

\subsection{The omnipresent arc graph and the grand arc graph}
The omnipresent arc graph $\mathcal{O}(S)$ was defined by Fanoni--Ghaswala--McLeay in \cite{fanoni2020homeomorphic} and does not require $S$ to have a puncture. First, note that all of the surfaces $S$ we consider in this note have finitely many ends, and are therefore \emph{stable} as defined in \cite{fanoni2020homeomorphic}, which is a necessary condition for the results of that paper. When $S$ has finitely many ends, the vertices of the omnipresent arc graph are isotopy classes of arcs that go to and from distinct ends of $S$. 
The edges correspond to disjointness. When $S$ has at least three ends, the omnipresent arc graph is connected, $\delta$-hyperbolic, and has infinite diameter. 
When $S$ has two ends, the graph is infinite-diameter, but not hyperbolic. The grand arc graph, defined by Bar-Natan--Verberne in \cite{bar2021grand}, is constructed similarly. In general, these two graphs are not the same, but when $S$ has finitely many ends the grand arc graph is exactly the omnipresent arc graph.

Given an endperiodic map $t$ on a surface $S$ with at least $2$ ends, we can take $Y$ to be a sufficiently large subsurface separating all of the finitely many ends of $S$ (see Figure~\ref{fig:separating}) so that $Y$ satisfies condition (1) of Theorem~\ref{thm:intro_main}. Enlarging $Y$ if necessary, assume that $Y$ satisfies (2) of Theorem~\ref{thm:intro_main} as well.  In fact, when $S$ has at least $3$ ends, such a subsurface is nondisplaceable (see Lemma 4.1 of \cite{aougab2021isometry} for a proof of this fact), so (2) follows immediately from this fact as well. Given that $Y$ separates all of the ends of $S$, any arc to and from distinct ends must essentially intersect $Y$, and thus, $Y$ is a witness for $\mathcal{O}(S)$. Then, taking any $g \in \map(Y)$ satisfying (3) of Theorem~\ref{thm:intro_main} produces an endperiodic map $f = tg$ acting loxodromically on the omnipresent/grand arc graph. 

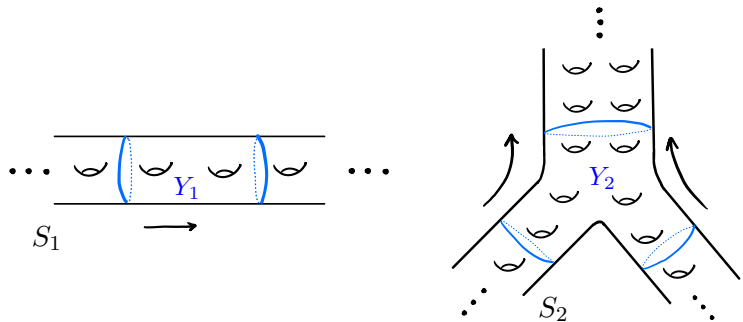
\begin{figure}[h]
\begin{center}
\begin{overpic}[trim = 1.25in 7.5in 1in 1.45in, clip=true, totalheight=0.23\textheight]{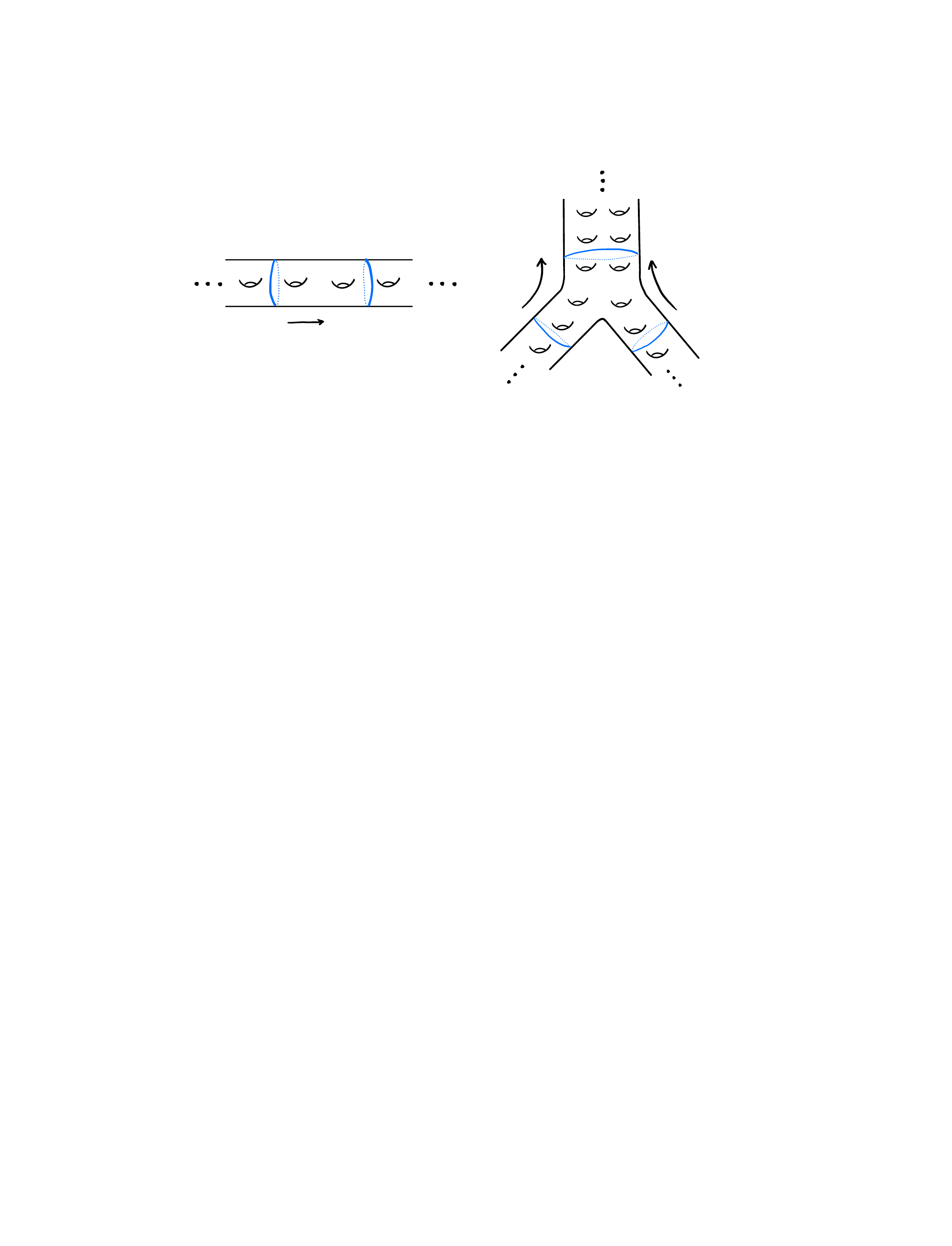}
\put(24,15){{\textcolor{blue}{$Y_1$}}}
\put(10,10){\large{$S_1$}}
\put(65,16){{\textcolor{blue}{$Y_2$}}}
\put(60,3){\large{$S_2$}}
\end{overpic}
\caption{The arrows in the figures indicate the direction of the endperiodic maps $t_i$ consisting of handleshifts in both cases. $Y_1$ and $Y_2$ are surfaces separating the ends of $S_1$ and $S_2$, respectively, that are compatible with $t_i$ in the sense that (1) and (2) of Theorem~\ref{thm:intro_main} are satisfied and $Y_i$ is a witness for $\mathcal{O}(S_i)$. }\label{fig:separating}
\end{center}
\end{figure}

We remark that the endperiodic map $f = tg$ constructed here is strongly irreducible in the sense of \cite{field2021end}. This essentially follows from \cite[Proposition 6.3]{field2021end}. A reasonable conjecture, suggested to us by Marissa Loving, is that the endperiodic maps that act loxodromically on $\mathcal{O}(S)$ are exactly the ones that are strongly irreducible.

\subsection{Separting curve graphs} 

In \cite{durham2018graphs}, Durham--Fanoni--Vlamis introduce the separating curve graphs $\sep(S, \mathcal{P})$ for certain infinite-type surfaces, where $\mathcal{P}$ is a finite collection of pairwise disjoint, closed subsets of the ends of $S$. These graphs are defined and interesting when $S$ has at least 4 ends and since $S$ has only finitely many ends in our setting, we can take $\mathcal{P} = e(S)$, i.e. each element of $\mathcal{P}$ is a singleton corresponding to an end of $S$. Then, the vertices of $\sep(S, \mathcal{P})$ correspond to isotopy classes of separating simple closed curves $c$ such that each component of $S \setminus c$ contains at least two ends of $S$. Edges in $\sep(S,\mathcal{P})$ correspond to disjointness, except in the case where $|e(S)| = 4$, where edges denote intersection number at most $2$. Note that $\sep(S, \mathcal{P})$ is an induced subgraph of the curve graph of $S$ (when $|e(S)| \neq 4$). Theorem 1.8 of \cite{durham2018graphs} states that, when $ 4 \leq |e(S)| < \infty$, $\sep(S, \mathcal{P})$ is connected, infinite-diameter, and $\delta$-hyperbolic. 

Given an endperiodic map $t$ on a surface $S$ with at least $4$ ends, we can take $Y$ to be a sufficiently large subsurface separating all of the finitely many ends of $S$, just as in the previous example, so that $Y$ satisfies condition (1) of Theorem~\ref{thm:intro_main}. Such a subsurface is nondisplaceable so (2) follows immediately from this fact. Moreover, \cite[Lemma 7.2]{durham2018graphs} shows that $Y$ is a witness for $\sep(S, \mathcal{P})$. Then, taking any $g \in \map(Y)$ satisfying (3) of Theorem~\ref{thm:intro_main} produces an endperiodic map $f = tg$ acting loxodromically on the separating curve graph $\sep(S, \mathcal{P})$.

\bibliographystyle{abbrv}
\bibliography{Endperiodic_lox.bbl}

\end{document}